\documentclass[12pt,reqno] {amsart}%
\usepackage{amsfonts}
\usepackage{amsmath}
\usepackage{amssymb}
\usepackage{xypic}
\usepackage{graphicx}%
 \theoremstyle{plain}

\newtheorem{theorem}{Theorem}

\newtheorem{lemma}[theorem]{Lemma}

\def\ds{\displaystyle}

\def\<{\langle}
\def\>{\rangle}

\begin{document}

\begin{center}
\textbf{General divisor functions in arithmetic progressions to large moduli}
\end{center}
\bigskip

\begin{center}
Fei Wei, Boqing Xue and Yitang Zhang
\end{center}

\bigskip


\textit{Introduction. } Let $a$ and $d>0$ denote integers with $(a,d)=1$ .  Given  an arithmetic function  $f$, the study of  the distribution of $f$ in the residue class $a(\text{mod }d)$  is  reduced  to estimating the quantity
\[
\Delta_f(x;d,a)=\ds\sum_{\substack{n\leq x\\n\equiv a(d)}}f(n)-\frac{1}{\varphi(d)}\ds\sum_{\substack{n\leq x\\(n,d)=1}}f(n).\eqno(1)
\]

Let $\tau_k$ denote the $k$-fold divisor function. Given $k$ and  $B>0$,  for fairly general  arithmetic functions $f$ satisfying
\[
f(n)\ll\tau_k(n)(1+\log n)^B,
\]
 a reasonable non-trivial upper  bound for $\Delta_f(x;d,a)$ is
\[
\Delta_f(x;d,a)\ll \frac{1}{\varphi(d)}\frac{x}{(\log x)^A}\eqno(2)
\]
 with $d$ lying in a certain range depending on $x$,  where $A$ is an arbitrarily large constant. For example, if  $f=\Lambda$, the von Mangoldt function,  then (2) is valid uniformly in the range $d<(\log x)^B$ (the Siegel-Walfisz theorem), while if
$f=\tau_k$, then (2) is valid uniformly in the range
  $d\leq x^{\vartheta_{k}-\varepsilon}$, where (see[2], [4], [5], [6] and [8])
\[
\vartheta_{2}=\frac{2}{3},\,\, \vartheta_{3}=\frac{21}{41},\,\, \vartheta_{4}=\frac{1}{2},\,\,\vartheta_{5}=\frac{9}{20},\,\,\vartheta_{6}=\frac{5}{12},\,\,\vartheta_{k}=\frac{8}{3k}\,\,(k\geq 7).
\]

Many problems in analytic number theory are reduced to proving the mean-value estimate
\[
\ds\sum_{d<x^{\vartheta}}\ds\max_{(a,d)=1}|\Delta_{f}(x;d,a)|\ll\frac{x}{(\log x)^A},\eqno(3)
\]
which shows that (2) holds for $d<x^{\vartheta}$ on average.  The celebrated  Bombieri-Vinogradov theorem asserts that (3) holds for $f=\Lambda$ with $\vartheta=1/2-\varepsilon$.
It is now known that, with  $\vartheta=1/2-\varepsilon$, (3) is valid for a large class of arithmetic functions, including all the divisor functions.

It is conjectured that (3) is valid for fairly general arithmetic functions $f$ with $\vartheta=1-\varepsilon$. However, for most functions $f$,  it seems extremely difficult to prove that (3)
holds with some $\vartheta$ greater than $1/2$. For example, if  $f=\tau_3$, the best result is that (3) holds with $\vartheta=11/21-\varepsilon$ (see  [4]), while one is still unable to prove (3) with any $\vartheta>1/2$ if $f=\tau_k$, $k\ge 4$.

In the recent work of  the third author on the bounded gaps between  primes [9], a crucial step is to show that, for any fixed $a\ne 0$,
\[
\ds\sum_{\substack{d<x^{\vartheta}\\d|\mathcal{P}(x^\varpi)\\(d,a)=1}}\ds|\Delta_{\Lambda}(x;d,a)|\ll\frac{x}{(\log x)^A}\eqno(4)
\]
with
 $$(\vartheta,\,\varpi)=\bigg(\frac{293}{584},\,\frac{1}{1168}\bigg),$$ where
 $$\mathcal{P}(y)=\prod_{p\leq y}p.$$
The values of $\vartheta$ and $\varpi$ have been improved by Polymath  [7]. In fact, both [7] and [9] deal with  somewhat more general forms. The purpose of this paper is to show that
the methods in [9] equally  apply to the $\tau_k$  in place of $\Lambda$. Our main result is

\begin{theorem}\label{thm1}
Let $\vartheta$ and $\varpi$ be as above.   For any  $k\geq 4$ and $a\neq 0$ we have
\[
\ds\sum_{\substack {d<x^{\vartheta}\\( d,\mathcal{P}(x^{\varpi}))>x^{1/8-4\varpi}\\ (d,a)=1}}\mu(d)^2\left|\Delta_{\tau_k}(x;d,a)\right|\ll x\exp{\left(-(\log x)^{1/2}\right)},\eqno(5)
\]
where the implied constant depends on $k$ and $a$.
\end{theorem}

Theorem \ref{thm1} admits several refinements. First, using the results of [7], it is likely to show that (5) is valid for any $\vartheta$ and $\varpi$ satisfying
\[
43\bigg(\vartheta-\frac12\bigg)+27\varpi<1
\]
(see [3, Page 185]); second, the factor $\mu(d)^2$  in  (5) that restricts the summation over square-free moduli could be removed; third, with extra effort,  it is even possible to replace the right side of (1) by $x^{1-\eta}$ for some $\eta>0$ depending on $k$, $\vartheta$ and $\varpi$.

  Forl $k\geq 4$, let  $N_i$, $1\le i\le k$, be such that
$$1\leq N_{k}\leq N_{k-1}\leq\ldots \leq N_{1}$$ and
$$x\leq N_{k}N_{k-1}\cdot\cdot\cdot N_{1}< 2x.$$ Let $\gamma=\beta_{1}\ast\beta_{2}\ast\ldots\ast\beta_{k}$ where  $\beta_{i}$ is   given by
\[
\beta_i(n)=
\begin{cases}
$1$, \quad & \text{if } n\in [N_i,(1+\rho) N_i),\\
$0$, & \text{otherwise}
\end{cases}
\]
with $\rho=\exp(-(\log x)^{7/12})$. Note that $\gamma$ is supported on $[x,\,3x)$.

 Using the arguments in [9, Section 6], we deduce  Theorem \ref{thm1}  from the following

\begin{theorem}\label{thm2}
With the same  notation as above, we have
\[
\ds\sum_{\substack {d<x^{\vartheta}\\( d,\mathcal{P}(x^{\varpi}))>x^{1/8-4\varpi}\\ (d,a)=1}}\mu(d)^2\left|\Delta_{\gamma}(3x;d,a)\right|\ll x\exp{\left(-(\log x)^{2/3}\right)}.\eqno(6)
\]
\end{theorem}

\textit{Combinatorial arguments. }Write $N_{i}=x^{\nu_{i}}$, so that
\[
0\leq \nu_{k}\leq \ldots \leq\nu_{1}\eqno(7)
\]
and
\[
 1\leq \nu_{k}+\cdot\cdot\cdot+\nu_{1}<1+\frac{\log 2}{\log x}.\eqno(8)
\]
\medskip\noindent

\begin{lemma}  \label{lemma1}
Suppose that
$$
\nu_1<\frac58-8\varpi,\eqno(9)
$$
and, for  any subset $I$ of $\{1,2,...,k\}$ , that
\[
\sum_{i\in I}\nu_i\notin\bigg[\frac{3}{8}+8\varpi ,\, \frac{5}{8}-8\varpi\bigg].\eqno(10)
\]
Then we have
\[
 \nu_{k}+\cdot\cdot\cdot+\nu_{4}+\nu_1<\frac{3}{8}+8\varpi+\frac{\log 2}{\log x}.\eqno(11)
\]
\end{lemma}

\begin{proof} By virtue of (8), it suffices to show that
\[
\nu_3+\nu_2>\frac58-8\varpi.\eqno(12)
\]

By (7),  (9) and (10) we have
\[
\nu_2\le\nu_1<\frac38+8\varpi.
\]
It follows by (8) that there exists an index $l\in[3,k]$ such that
\[
\nu_l+\nu_{l-1}+\cdot\cdot\cdot+\nu_2> \frac{5}{8}-8\varpi
\]
and
\[
\nu_{l-1}+\cdot\cdot\cdot+\nu_2< \frac{3}{8}+8\varpi,
\]
so that
\[
\nu_l>\frac14-16\varpi.
\]
Hence, by (7),
\[
\nu_3+\nu_2>\frac12-32\varpi>\frac38+8\varpi.
\]
This, together with (10), leads to (12).

\end{proof}
\medskip\noindent

For notational simplicity we write
\[
\mathcal{P}_1=\mathcal{P}(x^{\varpi}),\qquad \mathcal{P}_0=\mathcal{P}(\exp\{(\log x)^{1/4}\}).
\]
\medskip\noindent

\begin{lemma} \label{lemma2}

Let $\varepsilon>0$ be an arbitrarily small constant. Assume
\[
x^{1/2-\varepsilon}<d<x^{\vartheta},\qquad \mu(d)^2=1,\eqno(13)
\]
\[
(d,\,\mathcal{P}_0)<x^{\varpi}\eqno(14)
\]
and
\[
(d,\mathcal{P}_1)>x^{1/8-4\varpi}.\eqno(15)
\]
Then,  for any $R$ satisfying $x^\varpi\le R\le x^{44\varpi}$ or $x^{3/8+6\varpi}\le R\le x^{1/2-3\varpi}$,  there is a divisor $r$ of $d$ such that
\[
R<r<x^{\varpi}R
\]
and
\[(d/r,\,\mathcal{P}_0)=1.
\]
\end{lemma}
\medskip\noindent

\begin{proof} Since $d$ is square-free, we may write $d=d_0d_1d_2$ with
\[
d_0=(d,\mathcal{P}_0),\qquad d_1=\frac{(d,\mathcal{P}_1)}{(d,\mathcal{P}_0)},\qquad d_2=\prod_{\substack{p|d\\p>x^{\varpi}}}p.
\]
Note that $\vartheta=1/2+2\varpi$. By  (13), (14) and (15) we have $d_0<x^{\varpi}$ and $d_2<x^{3/8+6\varpi}$ . Thus, in a way similar to the proof of [9,\,Lemma 4], it can be verified that, for any
$R\in[x^\varpi, x^{44\varpi}]$, $d_1$ has a divisor $r_1$ such that
\[
R<d_0r_1<x^\varpi R,
\]
and, for any  $R\in[x^{3/8+6\varpi}, x^{1/2-3\varpi}]$, $d_1$ has a divisor $r_2$ such that
\[
R<d_0d_2r_2<x^\varpi R.
\]
The assertion therefore follows by choosing $r=d_0r_1$ and $r=d_0d_2r_2$ respectively.
\end{proof}
\medskip\noindent

\begin{proof} [Proof of Theorem \ref{thm2}] First we note that, on the left-hand side of  (6), the contribution from the terms with $d<x^{1/2-\varepsilon}$  is acceptable. This is easily proved via the large sieve inequality (see [1,\,Theorem 0]). We can write $\gamma=\alpha*\beta$ with $\beta=\beta(n)$ supported on $[N,\,2N)$ for some $N$,
and verify that  for any $q$,  $r$ and $a$ satisfying $(a,r)=1$,
\[
\sum_{\substack {n\equiv a(r)\\ (n,q)=1}}\beta(n)-\frac{1}{\varphi(r)}\ds\sum_{(n,qr)=1}\beta(n)\ll x^{-\kappa}N\eqno (16)
\]
with some constant $\kappa>0$. This bound, which  is much sharper than  the ``Siegel-Walfisz" assumption, is used for bounding the errors arising from small moduli.

In order to prove Theorem \ref{thm2}, on the left-hand side of (6), we can replace the range $d<x^{\vartheta}$ by  $x^{1/2-\varepsilon}<d<x^{\vartheta}$, and impose the constraint  (14).

The proof is divided into three cases:

Case (a).  $\nu_1\ge 5/8-8\varpi$ .

Case (b). $\nu_1<5/8-8\varpi$ and (10) is valid for any $I\subset\{1,2,...,k\}$.

Case (c). There is a subset $I$ of $\{1,2,...,k\}$ such that
\[
\sum_{i\in I}\nu_i\in\bigg[\frac{3}{8}+8\varpi ,\, \frac{5}{8}-8\varpi\bigg].
\]
\medskip\noindent

{Proof in Case (a)}.\, Since
$$
\frac58-8\varpi>\vartheta,
$$
for any $d<x^{\vartheta}$ we have trivially
$$
\Delta_{\gamma}(3x;d,a)\ll\frac{1}{\varphi(d)}x\exp\{-(\log x)^{-2/3}\}.\eqno(17)
$$
This leads to (6).
\medskip\noindent

{Proof in Case (b)}.\, Write  $\alpha=\beta_{4}\ast\ldots\ast\beta_{k}$ so that $\gamma=\alpha\ast\beta_1\ast\beta_2\ast\beta_3$. Let
\[
M=\prod_{4\le i\le k}N_i.
\]
Note that $\alpha\ast\beta_1$  is supported on $[MN_1,\,2MN_1)$ and, by Lemma \ref{lemma1}, $MN_1\ll x^{3/8+8\varpi}$. Hence,  by the Type III estimate in [9] , we find that (17) is valid for any $d$ satisfying
(13), (14) and (15).
Here we have also used Lemma \ref{lemma2}.  For details the reader is referred to the last two sections of [9].
\medskip\noindent

{Proof in Case (c)}. Without loss of generality, we can assume that there is a subset $I$  of   $\{1,2,...,k\}$ such that
\[
\frac38+8\varpi<\sum_{i\in I}\nu_i<\frac12+\frac{\log 2}{2\log x}.\eqno(18)
\]

Let $J$ be the complement of $I$ in  $\{1,2,...,k\}$. Write
\[
 J=\{j_1,\,j_2,...,j_m\},\qquad I=\{i_1,\,i_2,...,i_l\},
\]
\[
\alpha=\beta_{j_1}\ast\beta_{j_2}\ast...\ast\beta_{j_m}\qquad\text{and}\qquad \beta=\beta_{i_1}\ast\beta_{i_2}\ast...\ast\beta_{i_l},
\]
so that $\gamma=\alpha\ast\beta$. Note that $\alpha$ and $\beta$ are supported on $[M,2M)$ and $[N,2N)$ respectively, where
\[
M=\prod_{j\in J}N_j,\qquad N=\prod_{i\in I}N_i.
\]
By (18) we have
\[
x^{3/8+8\varpi}<N\ll x^{1/2}.
\]
By Lemma \ref{lemma2}, the proof is now reduced to showing that
\[
\sum_{\substack{Q\le q<2Q\\(q,a\mathcal{P}_0)=1}}\,\sum_{\substack{R\le r<2R\\(r,a)=1}}\,\mu(qr)^2|\Delta_\gamma(3x;qr,a)|\ll  x\exp{\left(-2(\log x)^{2/3}\right)}\eqno(19)
\]
with $Q$ and $R$ satisfying
$x^{1/2-\varepsilon}\ll QR\ll x^{\vartheta}$,
\[
x^{-\varpi-\varepsilon} N\ll R\ll x^{-\varepsilon}N\quad\text{if}\quad N<x^{1/2-4\varpi},
\]
 and
\[
x^{-4\varpi} N\ll R\ll x^{-3\varpi}N\quad\text{if}\quad N\ge x^{1/2-4\varpi}.
\]

By the Type I and II estimates in [9] we obtain (19). Here we apply the estimate (16) for bounding the errors arising from small moduli in place of the  ``Siegel-Walfisz" assumption. For details the reader is referred to Section 7-12 of [9].
\end{proof}

\end{document}